\documentclass[11pt]{amsart}

\usepackage{amssymb}
\usepackage{bm}
\usepackage{graphicx}
\usepackage{psfrag}
\usepackage[usenames,dvipsnames]{xcolor}
\usepackage{enumerate}
\usepackage{multirow}
\usepackage{url}

\usepackage{comment}

\usepackage{algpseudocode}

\usepackage{mathtools}

\usepackage{xy}
\input xy
\xyoption{all}

\newcommand{\excise}[1]{}

\newtheorem{thm}{Theorem}[section]
\newtheorem{lemma}[thm]{Lemma}

\newtheorem{cor}[thm]{Corollary}

\theoremstyle{definition}

\numberwithin{equation}{section}

\newcommand\NN{\mathbb{N}}

\begin{document}

\mbox{}
\title[Unboundedness of irreducible decompositions]{Unboundedness of irreducible decompositions of numerical semigroups}

\author[Bogart]{Tristram Bogart}
\address{Departamento de Matem\'aticas \\ Universidad de los Andes \\ Bogot\'a, Colombia}
\email{tc.bogart22@uniandes.edu.co}

\author[Seyed Fakhari]{Seyed Amin Seyed Fakhari}
\address{Departamento de Matem\'aticas \\ Universidad de los Andes \\ Bogot\'a, Colombia}
\email{s.seyedfakhari@uniandes.edu.co}
\date{\today}

\begin{abstract}
We present two families of numerical semigroups and show that for each family, the number of required components in an irreducible decomposition cannot be bounded by any given integer. This gives a negative answer to a question raised by Delgado, Garc\'ia-S\'anchez and Rosales.
\end{abstract}

\maketitle


\section{Introduction}\label{sec:intro}
Let $\mathbb{N}$ denote the set of nonnegative integers. A subset $S\subseteq \mathbb{N}$ is called a {\it numerical semigroup} if (i) $S$ is closed under addition, (ii) $0\in S$ and (iii) $\mathbb{N}\setminus S$ is a finite set. Any element in $\mathbb{N}\setminus S$ is called a {\it gap} of $S$ and the largest gap of $S$ is the {\it Frobenius number} of $S$, denoted by $F(S)$. For a nonempty subset $A\subseteq \mathbb{N}$, set $$\langle A\rangle=\{c_1x_1+\cdots +c_nx_n\mid n\in \mathbb{N}, c_i\in \mathbb{N}, x_i\in A\}.$$ Then $\langle A\rangle$ is a numerical semigroup if and only if the greatest common divisor ${\rm gcd}(A)$ is equal to one. In this case $\langle A\rangle$ is called the {\it numerical semigroup generated by} $A$, and $A$ is called a {\it generating set} of $\langle A\rangle$. It is well-known that any numerical semigroup admits a unique minimal (under inclusion) generating set. Let $S$ be a numerical semigroup and let $A$ be the unique minimal generating set of $S$. Then the smallest integer in $A$ is called the {\it multiplicity} of $S$.

A numerical semigroup $\Gamma$ is {\it irreducible} if it cannot be written as the intersection of two numerical semigroups which properly contain $\Gamma$. It turns out (\cite[Proposition 4.44]{rg}) that any numerical semigroup $S$ can be decomposed as the intersection of finitely many irreducible numerical semigroups. Such a decomposition is called an {\it irreducible} decomposition of $S$. The irreducible numerical semigroups appearing in such a decomposition are called the {\it irreducible components} of the intersection. In \cite[Page 4]{dgr}, the authors ask whether there is an integer $N$ such that every numerical semigroup has an irreducible decomposition with at most $N$ irreducible components. The goal of this paper is to give a negative answer to this question.

Indeed, in Section \ref{sec:constructions}, we construct two distinct families of numerical semigroups and show that for each family, the number of required components in an irreducible decomposition is unbounded. The first family has ``low'' multiplicity compared to the Frobenius number, while the second, being the semigroup $\NN^{\geq n}$ for certain values of $n$, has the largest possible multiplicity for a given Frobenius number.

\section{Two Constructions} \label{sec:constructions}
In this section we provide two families of numerical semigroups and  prove that for each family, the number of required components in a irreducible decomposition is unbounded. But first, we need to mention some definitions and basic facts.

Let $S$ be a numerical semigroup. If $F(S)$ is an odd integer then $S$ is called {\it symmetric} whenever $x\in \mathbb{N}\setminus S$ implies $F(S)-x\in S$. Similarly, if $F(S)$ is an even integer then $S$ is called {\it pseudo-symmetric} whenever $x\in \mathbb{N}\setminus (S\cup \{\frac{F(S)}{2}\})$ implies $F(S)-x\in S$. We know from \cite[Proposition 2]{rb2} that a numerical semigroup is irreducible if and only if it is either symmetric or pseudo-symmetric.

Let $S$ be a numerical semigroup. An integer $x\in \mathbb{Z}$ is called a {\it pseudo-Frobenius number} of $S$ if $x\notin S$ but $x+s$ belongs to $S$ for each nonzero element $s\in S$. The set of pseudo-Frobenius numbers of $S$ will be denoted by $PF(S)$. It is obvious that $F(S)\in PF(S)$. Define $$BPF(S):=\left\{a\in PF(S)\mid a > \frac{F(S)}{2}\right\}.$$ Moreover, for each element $a\in BPF(S)$, set$$\xi(a):=\{a+t \mid t\in \mathbb{N} \ {\rm and} \ a+t \notin \langle S,t\rangle\}.$$Rosales and Branco \cite{rb} determined a lower bound and an upper bound for the minimum number of components in an irreducible decomposition of $S$. Here is their result.

\begin{thm} [\cite{rb}, Corollaries 10 and 18] \label{bounds}
Let $S$ be a numerical semigroup and suppose $BPF(S)=\{a_1, \ldots, a_m\}$ Then the following statements hold.
\begin{itemize}
\item[(i)] $S$ admits an irreducible decomposition with at most $m$ irreducible components.
\item[(ii)] Set $h(S):=\min\big\{\#\{b_1, \ldots, b_m\}\mid (b_1, \ldots, b_m)\in \xi(a_1)\times \cdots\times \xi(a_m)\big\}$. Then every irreducible decomposition of $S$ has at least $h(S)$ irreducible components.
\end{itemize}
\end{thm}

We are now ready to present our first construction.

\begin{thm} \label{thm:lowmultiplicity}
Let $k \geq 2$ be an integer and suppose that $d$ is the smallest prime divisor of $k$. For any integer $n \geq k^2-1$ such that $k$ does not divide $n-1$, consider the numerical semigroup
\[ S_{k,n} = \langle k, n, n+1, \dots, n+k-1 \rangle. \]
Then every irreducible decomposition of $S_{k,n}$ has at least $d-1$ irreducible components.
\end{thm}

Note that the Frobenius number of $S_{k,n}$ is $n-1$. In order to prove Theorem \ref{thm:lowmultiplicity}, we need the following lemma.

\begin{lemma} \label{irredgap}
Let $I$ be an irreducible semigroup that contains $S_{k,n}$ and $i \in \{1,2, \dots d-1\}$. If $n-i$ is a gap of $I$, then $F(I) = n-i$.
\end{lemma}

\begin{proof}
Since $I \supset S_{k,n}$, $F(I) \leq F(S_{k,n}) = n - 1$, and since $n-i$ is a gap of $I$, $F(I) \geq n-i$. So $F(I) = n - j$ for some $1 \leq j \leq i$. By symmetry or pseudosymmetry, we must have
  \[ F(I) - (n-i) = (n-j)-(n-i) = i-j \in I .\]

  If $j < i$, then $1 \leq i-j \leq d-2$. In particular, since $d$ is the smallest prime divisor of $k$, $i-j$ and $k$ are coprime. Then since $i-j \in I$ and $k \in S_{k,n} \subseteq I$, we have $F(I) \leq F \left( \langle i-j, k \rangle \right)$.

  Now for any coprime pair $a, b$, it is well-known (\cite[Proposition 2.13]{dgr}) that $F(\langle a, b \rangle) = ab-a-b$. Thus
  \[ F(I) \leq (i-j)k - (i-j) - k  < (i-j)k \leq (k-2)k = k^2-2k \]
which is impossible since $F(I) = n-1 \geq k^2-2$. Thus $j$ must equal $i$.
\end{proof}

Using Lemma \ref{irredgap}, we are able to prove Theorem \ref{thm:lowmultiplicity}.

\begin{proof}[Proof of Theorem \ref{thm:lowmultiplicity}]
Fix $1 \leq i \leq d-1$. Any irreducible decomposition of $S_{k,n}$ includes at least one irreducible $I_i \supset S_{k,n}$ such that $n-i$ is a gap of $I$. By the lemma, $F(I_i) = n-i$. Thus the irreducible components $I_1, I_2, \dots, I_{d-1}$ are all distinct.
\end{proof}

One case of Theorem \ref{thm:lowmultiplicity}, that maximizes the number of irreducible components while nearly minimizing the Frobenius number, occurs when $k=p$ is prime and $n = p^2$. Furthermore, the lower bound on the number of components turns out to be sharp in this case.

\begin{cor} Let $p$ be a prime and consider the numerical semigroup
  \[ S = \langle p, p^2+1, p^2+2, \dots, p^2+p-1 \rangle \]
  of Frobnius number $p^2-1$. Then the minimum number of components in an irreducible decomposition of $S$ is exactly $p-1$.
\end{cor}

\begin{proof}
The lower bound follows immediately from Theorem \ref{thm:lowmultiplicity} and the upper bound follows from Theorem \ref{bounds}(i).
\end{proof}

Next, we present our second construction.

\begin{thm}
For any integer $k$ there is an integer $n$ such that the number of irreducible components in any irreducible decomposition of $S=\{n+1, n+2, \ldots\}$ is at least $k$.
\end{thm}

\begin{proof}
Note that $\{1, 2, \ldots, n\}$ is the set of pseudo-Frobenius numbers of $S$ and $F(S)=n$. This yields that
\begin{align*}
BPF(S)=\{\lfloor n/2\rfloor+1, \lfloor n/2\rfloor+2, \ldots, \}.
\end{align*}

We define a sequence $a_1, \ldots, a_k$ as follows. Set $a_1:=1$ and $a_i:=a_{i-1}!+a_{i-1}$ for each integer $i=2, \ldots, k$. We may choose $n$ large enough such that $n-a_k\geq \lfloor n/2\rfloor+1$ (and so $n-a_1, \ldots, n-a_k\in BPF(S)$.) Moreover, we may assume that $n-a_k$ is divisible by $a_k!$. For each $i=1,2, \ldots k$, consider the set$$\xi(n-a_i)=\{n-a_i+t: t\in \mathbb{N} \ {\rm and} \ n-a_i+t\notin \langle S,t\rangle\}.$$Using a backward induction, we show that $\xi(n-a_i)$ is a singleton for each $i=1,2, \ldots k$. First, suppose that $i=k$. If there is a nonzero number $t\in \mathbb{N}$ such that $n-a_k+t\notin \langle S,t\rangle$, then $t\leq a_k$ and $n-a_k$ is not divisible by $t$. This contradicts the choice of $n$ (as $n-a_k$ is divisible by any natural number less than or equal to $a_k$). Thus $t=0$ and so $\xi(n-a_i)=\{n-a_i\}$. Now, assume that $i\leq k-1$. If $n-a_i+t\notin \langle S,t\rangle$, for some $t\in \mathbb{N}\setminus \{0\}$, then $t\leq a_i$ and $n-a_i$ is not divisible by $t$. It follows from $t\leq a_i$ that $t$ divides $a_i!$. Since $n-a_i$ is not divisible by $t$, we deduce that $n-a_i-a_i!=n-a_{i+1}$ is not divisible by $t$. Hence, $n-a_{i+1}+t\notin \langle S,t\rangle$. Consequently, $n-a_{i+1}+t\in \xi(n-a_{i+1})$ which contradicts the induction hypothesis. Therefore, $t=0$ and $\xi(n-a_i)=\{n-a_i\}$. Consequently, the quantity $h(S)$, as defined in Theorem \ref{bounds}(ii), is at least $k$ and the assertion follows from that theorem.
\end{proof}

\section{Acknowledgements}
Seyed Amin Seyed Fakhari was supported by his FAPA grant from Universidad de los Andes.


\end{document}